\documentclass [reqno,10pt,oneside]{amsart}

\usepackage{amsthm}
\usepackage{amscd}
\usepackage{amsmath,amssymb}
\usepackage{algorithm}
\usepackage[noend]{algorithmic}

\theoremstyle {plain}
\newtheorem {thm}{Theorem}[section]
\newtheorem {prop}[thm]{Proposition}
\newtheorem {lem}[thm]{Lemma}
\newtheorem {cor}[thm]{Corollary}
\theoremstyle {definition}
\newtheorem {defn}[thm]{Definition}

\theoremstyle {remark}
\newtheorem {rem}[thm]{Remark}

\newtheorem {exmp}[thm]{Example}
\newtheorem{con}[thm]{Conclusion}

\DeclareMathOperator{\Mon}{Mon}
\DeclareMathOperator{\LM}{LM}
\DeclareMathOperator{\tail}{tail}
\DeclareMathOperator{\wdeg}{w-deg}
\DeclareMathOperator{\HF}{HF}

\newcommand{\C}{{\mathbb C}}

\newcommand{\F}{{\mathbb F}}
\newcommand{\N}{{\mathbb N}}
\newcommand{\Q}{{\mathbb Q}}

\newcommand{\Z}{{\mathbb Z}}

\newcommand{\gen}[1]{\left\langle #1 \right\rangle}
\newcommand{\set}[2]{\left\{{#1},\ldots,{#2}\right\}}
\newcommand{\singular}{{\sc Singular }}

\begin{document}

\bibliographystyle{alpha}

\title{Parallelization of Modular Algorithms}

\author{Nazeran Idrees}
\address{Nazeran Idrees\\ Abdus Salam School of Mathematical Sciences\\ GC University\\ 
Lahore\\ 68-B\\ New Muslim Town\\ Lahore 54600\\ Pakistan}
\email{nazeranjawwad@gmail.com}

\author{Gerhard Pfister}
\address{Gerhard Pfister\\ Department of Mathematics\\ University of Kaiserslautern\\
Erwin-Schr\"odinger-Str.\\ 67663 Kaiserslautern\\ Germany}
\email{pfister@mathematik.uni-kl.de}
\urladdr{http://www.mathematik.uni-kl.de/$\sim$pfister} 

\author{Stefan Steidel}
\address{Stefan Steidel\\ Department of Mathematics\\ University of Kaiserslautern\\ 
Erwin-Schr\"odinger-Str.\\ 67663 Kaiserslautern\\ Germany}
\email{steidel@mathematik.uni-kl.de}
\urladdr{http://www.mathematik.uni-kl.de/$\sim$steidel} 


\thanks{Part of the work was done at ASSMS, GCU Lahore -- Pakistan}

\date{\today}

\maketitle

\begin{abstract}
In this paper we investigate the parallelization of two modular algorithms. In fact, we
consider the modular computation of Gr\"obner bases (resp. standard bases) and the modular 
computation of the associated primes of a zero--dimensional ideal and describe their parallel 
implementation in \textsc{Singular}.
Our modular algorithms to solve problems over $\Q$ mainly consist of three parts, solving the 
problem modulo $p$ for several primes $p$, lifting the result to $\Q$ by applying Chinese 
remainder resp. rational reconstruction, and a part of verification.
Arnold proved using the Hilbert function that the verification part in the modular algorithm to
compute Gr\"obner bases can be simplified for homogeneous ideals (cf. \cite{A03}).
The idea of the proof could easily be adapted to the local case, i.e. for local orderings and
not necessarily homogeneous ideals, using the Hilbert--Samuel function (cf. \cite{Pf07}).
In this paper we prove the corresponding theorem for non--homogeneous ideals in case of
a global ordering.
\end{abstract}


\section{Introduction} \label{secIntro}
 
We consider an ideal in a polynomial ring over the rationals. In section \ref{secModStd} we 
describe a parallel modular implementation of the Gr\"obner basis (resp. standard basis) 
algorithm. Afterwards we restrict ourselves to the case of a zero--dimensional ideal
and introduce a parallel modular implementation of the algorithm to compute the associated 
primes in section \ref{secAssPrimes}. Finally we give a couple of examples with corresponding 
timings and some conclusions in section \ref{secExTime}. Both algorithms are implemented in 
\textsc{Singular}. The Gr\"obner basis resp. standard basis algorithm can be found in the library 
\texttt{modstd.lib} and the algorithm for computing the associated primes in 
\texttt{assprimeszerodim.lib}. They are included in the release \textsc{Singular 3-1-2}.

\vspace{0.2cm}

The task to compute a Gr\"obner basis $G$ of an ideal $I$ using modular methods consists of 
three steps. In the first step, we compute the Gr\"obner basis modulo $p$ for sufficiently many 
primes $p$ and, in the second step, use Chinese remainder and rational reconstruction to obtain 
a result over $\Q$. 
In the third step, we have to verify that the result obtained this way is correct, i.e. to verify that $I = 
\gen G$ and $G$ is a Gr\"obner basis of $\gen G$. If this fails we go back to the first step. The 
third step is usually at least as time consuming as the first step. Omitting the third step 
would produce a Gr\"obner basis only with high probability and the result could be wrong in 
extreme situations. It is known that some of the commercial computer algebra systems have 
problems in this direction.\footnote{Let $N$ be the product of all primes smaller than $2^{32}$ 
and $I = \langle v+w+x+y+z, vw+wx+xy+yz+vz, vwx+wxy+xyz+vyz+vwz, vwxy+wxyz+vxyz+vwyz
+vwxz, vwxyz+N \rangle \subseteq \Q[v,w,x,y,z]$. Then \textsc{Magma V2.16--11} (64-bit version) 
computes a wrong Gr\"obner basis, in particular it computes the Gr\"obner basis of the ideal $J 
= \langle v+w+x+y+z, vw+wx+xy+yz+vz, vwx+wxy+xyz+vyz+vwz, vwxy+wxyz+vxyz+vwyz+vwxz, 
vwxyz \rangle \subseteq \Q[v,w,x,y,z]$ which obviously differs from $I$.}

\vspace{0.2cm}

Arnold proved using the Hilbert function that the verification part in the modular algorithm to
compute Gr\"obner bases can be simplified for homogeneous ideals (cf. \cite{A03}): Let $I 
\subseteq \Q[x_1,\ldots,x_n]$ be a homogeneous ideal, $>$ a global monomial ordering and 
$G \subseteq \Q[x_1,\ldots,x_n]$ be a set of polynomials such that $I \subseteq \gen G$, $G$
is a Gr\"obner basis of $\gen G$ and $\LM(G) = \LM(I\F_p[x_1,\ldots,x_n])$ for some prime
number $p$ where $\LM(G)$ denotes the set of leading monomials of $G$ w.r.t. $>$, then
$G$ is a Gr\"obner basis of $I$.
The idea of the proof could easily be adapted to the local case, i.e. for local orderings and
$I$ not necessarily homogeneous, using the Hilbert--Samuel function (cf. \cite{Pf07}).
In this paper we prove the corresponding theorem for non--homogeneous ideals in case of
a global ordering.
Two important assumptions of the theorem are the facts that $I \subseteq \gen G$ and $G$ is 
a Gr\"obner basis of $\gen G$. This verification can be very time consuming in a negative case. 
Hence, we use a so--called \textsc{pTestSB} which is one of the new ideas for our algorithm. 
Therefore we randomly choose a prime number $p$ which has not been used in the previous 
computations and perform the verification modulo $p$. Only if the \textsc{pTestSB} is positive 
we perform the verification over $\Q$, and the last required condition that $\LM(G) = 
\LM(I\F_p[x_1,\ldots,x_n])$ is then automatically fulfilled.

\vspace{0.2cm}

The implementation of our algorithm as \singular library implies that we did not change the
kernel routines of \textsc{Singular}. We plan to implement the algorithm in the kernel of \singular in
future. For this purpose we can apply the ideas of Gr\"abe (cf. \cite{G94}) - using multimodular 
coefficients - and Traverso (cf. \cite{T89}) - using the trace--algorithm. The trace--algorithm would 
speed up the computations in positive characteristic a lot. We compute a Gr\"obner basis of an ideal 
$I \subseteq \Q[x_1,\ldots,x_n]$ over $\F_p[x_1,\ldots,x_n]$ for a random prime $p$ and keep in mind 
the zero--reductions of the $s$--polynomials such that we do not perform these reductions in any other
Gr\"obner basis computation over $\F_q[x_1,\ldots,x_n]$ for primes $q \neq p$. We do not need 
these information, i.e. the guarantee that we really obtain a Gr\"obner basis over $\F_q[x_1,\ldots,
x_n]$, since we have the verification step - that the lifted result over $\Q[x_1,\ldots,x_n]$ is a 
Gr\"obner basis of $I$ - at the end anyway. 

\vspace{0.2cm}

Our idea regarding the primary decomposition of a zero--dimensional ideal $I \subseteq \Q[X]$ is to
compute the associated primes $M_1,\ldots,M_s$ of $I$ and use seperators $\sigma_1,\ldots,
\sigma_s$\footnote{We call $\sigma_i$ a \emph{seperator} w.r.t. $M_i$ if $\sigma_i \notin M_i$ and 
$\sigma_i \in M_j$ for $j \neq i$.} such that the saturation $I:\sigma_i^\infty$ of $I$ w.r.t. $\sigma_i$ is 
the primary ideal corresponding to $M_i$ (cf. \cite{SY96}).
The computation of the associated primes is based on the so--called Shape Lemma (Proposition 
\ref{propCoordinateChange}(2)). Here, one new idea is to choose a generic linear form $r = a_1x_1 + 
\ldots + a_{n-1}x_{n-1} + x_n$ with $a_1,\ldots,a_{n-1} \in \Z$ and a random prime $p$ to test if $\dim_{\F_p}
(\F_p[X]/I\F_p[X]) = \dim_{\F_p}\big(\F_p[x_n]/(\psi(I)\F_p[X] \cap \F_p[x_n])\big)$, i.e. $\psi(I)\F_p[X] = 
\gen{x_1-g_1(x_n), \ldots, x_{n-1}-g_{n-1}(x_n), F(x_n)}$ whereat $\psi$ denotes the linear map defined by 
$\psi(x_i) = x_i$ for $i=1,\ldots,n-1$ and $\psi(x_n)=2x_n-r$. If this test called \textsc{pTestRad} is positive 
then the ideal $I$ in $\Q[X]$ has the same property with high probability.
If the test is negative then we compute the radical of $I$ using the idea of Krick and Logar (Proposition
\ref{propRad}(1)) combined with modular methods, and replace $I$ by $\sqrt I$. Afterwards we compute 
$\gen F = \langle I, T-r\rangle_{\Q[X,T]} \cap \Q[T]$, again using modular methods, i.e. we compute 
$F^{(p)}$ such that $ \langle F^{(p)} \rangle = \langle I, T-r \rangle_{\F_p[X,T]} \cap \F_p[T]$ and 
$\deg(F^{(p)}) = \dim_\Q(\Q[X]/I)$ for sufficiently many primes $p$, and we use Chinese remainder 
and rational reconstruction to obtain $F \in \Q[T]$. 
The verification is the test whether it holds $F(r) \in I$ and no proper factor of $F(r)$ is in $I$. If $F = 
F_1^{\nu_1} \cdots F_s^{\nu_s}$ is the factorization of $F$ in $\Q[T]$ into irreducible factors then $M_1 = 
\langle I, F_1(r)\rangle, \ldots, M_s = \langle I, F_s(r)\rangle$ are the associated primes of $I$.
The new ideas in this approach are the \textsc{pTestRad} described above and the fact that we do not
compute the associated primes in positive characteristic but instead one special generator of the radical,
$F(r)$, which is much better to control.\footnote{The computation of the associated primes in positive
characteristic would create similar problems as the factorization of polynomials: Different behaviour of
splitting in different characteristics. Therefore it is easier and faster to compute $F \in \Q[T]$ and factorize
this polynomial.}

\vspace{0.2cm}

We use the following notation. Let $X=\{x_1,\ldots,x_n\}$ be a set of variables. We denote by 
$\Mon(X)$ the set of monomials, and by $\Q[X]$ the polynomial ring over $\Q$ in these $n$ 
indeterminates. Let $S \subseteq \Q[X]$ be a set of polynomials, then $\LM(S) := \{\LM(f) \mid f 
\in S\}$ is the set of leading monomials of $S$. Given an ideal $I \subseteq \Q[X]$ we can always 
choose a finite set of polynomials $F_I$ such that $I = \gen{F_I}$. 
If $I =\gen{f_1,\ldots,f_r} \subseteq \Q[X]$ and $p$ is a prime number which does not divide any 
denominator of the coefficients of $f_1,\ldots,f_r$ we will write $I_p := \gen{f_1 \mod p,\ldots,f_r 
\mod p} \subseteq \F_p[X]$.

\section{Computing Gr\"obner bases using modular methods} \label{secModStd}

In the following we consider an ideal $I = \gen{f_1,\ldots,f_r} \subseteq \Q[X]$ together with a 
monomial ordering $>$ and set $F_I =\set{f_1}{f_r}$. We assume that $>$ is either global or local. 
Within this section we describe an algorithm for computing a Gr\"obner basis resp. a standard 
basis\footnote{For definitions and properties cf. \cite{GP07}.} $G \subseteq \Q[X]$ of $I$ by using 
modular methods. 

The basic idea of the algorithm is as follows. Choose a set $P$ of prime numbers, compute 
standard bases $G_p$ of $I_p \subseteq \F_p[X]$, for every $p \in P$, and finally lift these 
modular standard bases to a standard basis $G \subseteq \Q[X]$ of $I$. The lifting process 
consists of two steps. Firstly, the set $GP:=\{G_p \mid p \in P\}$ is lifted to $G_N 
\subseteq \Z/N\Z[X]$ with $N:=\prod_{p \in P} p$ by applying the Chinese remainder algorithm 
to the coefficients of the polynomials occuring in $GP$. Since $G_N$ is uniquely determined 
modulo $N$, theory requires $N$ to be larger than the moduli of all coefficients occuring in a 
standard basis of $I$ over $\Q$. This issue is not revisable a priori and will be discussed later in 
this section. Secondly, we obtain $G \subseteq \Q[X]$ by pulling back the modular coefficients 
occuring in $G_N$ to rational coefficients via the Farey rational map\footnote{\emph{Farey 
fractions} refer to rational reconstruction. A definition of \emph{Farey fractions} and the 
\emph{Farey rational map} can be found in \cite{A03},\cite{KG83},\cite{Pf07}; for remarks 
concerning its computation cf. \cite{KG83}.}. 
This map is guaranteed to be bijective provided that $\sqrt{N/2}$ is larger than the moduli
of all coefficients in $G$.\footnote{Remarks on the required bound on the coefficients are 
given in \cite{KG83}.} The latter condition on $N$ concerning the Farey rational map obviously 
implies the former condition concerning the Chinese remainder algorithm. We consequently 
define two corresponding notions that are essential regarding the algorithm.

\begin{defn} \label{defnLucky}
Let $G$ be a standard basis of $I$.
\begin{enumerate}
\item If $G_p$ is a standard basis of $I_p$, then the prime number $p$ is called \emph{lucky 
          for $I$} if and only if $\LM(G) = \LM(G_p)$. Otherwise $p$ is called \emph{unlucky for $I$}.
\item A set $P$ of lucky primes for $I$ is called \emph{sufficiently large for $I$} if and only if 
          $\prod_{p \in P} p \geq \max\{2 \cdot |c|^2 \mid c \,\text{ coefficient occuring in }\, G\}$. 
\end{enumerate}
\end{defn}

Now we can concretize the theoretical idea of the algorithm. Consider a sufficiently large 
set $P$ of lucky primes for $I$ such that none of these primes divides any coefficient 
occuring in $F_I$, compute the set $GP$, and lift this result to a rational standard basis 
$G$ of $I$ as aforementioned. More details can be found in \cite{A03}.

In practice, we have to handle two difficulties since naturally the standard basis $G$ of $I$ 
is a priori unknown. In fact, it is necessary to ensure that every prime number used is lucky 
for $I$, and to decide whether the chosen set of primes is sufficiently large for $I$. 

Therefore, we fix a natural number $s$ and an arbitrary set of primes $P$ of cardinality $s$. 
After having computed the set of standard bases $GP:=\{G_p \mid p \in P\}$ we delete the 
unlucky primes in the following way.

\vspace{0.2cm}

\emph{\textsc{deleteUnluckyPrimesSB:} We define an equivalence relation on 
$(GP,P)$ by \linebreak $(G_p,p) \sim (G_q,q) :\Longleftrightarrow \LM(G_p) = \LM(G_q)$. 
Then the equivalence class of largest cardinality is stored in $(GP,P)$, the others are deleted.} 

\vspace{0.2cm}

With the aid of this method we are able to choose a set of lucky primes with high 
probability. A faulty decision will be compensated by subsequent tests. 

Since we cannot predict if a given set of primes $P$ is sufficiently large for $I$, we have 
to proceed by trial and error. Hence, we lift the set $GP$ to $G \subseteq \Q[X]$, as per the
description at the beginning of this section, and test whether $G$ is already a standard 
basis of $I$. Otherwise we enlarge the set $P$ by $s$ new prime numbers and continue 
analogously until once the test is positive. The test especially verifies whether $G$ is a 
standard basis of $\gen G$, but this computation in $\Q[X]$ can be very expensive if 
$P$ is far away from being sufficiently large for $I$. Hence, we prefix a test in positive 
characteristic that is a sufficient criterion if $P$ is not sufficiently large for $I$.

\vspace{0.2cm}

\emph{\textsc{pTestSB:} We randomly choose a prime number $p \notin P$ such
that $p$ does not divide the numerator and denominator of any coefficient occuring in $F_I$. 
The test is positive if and only if $(G \mod p)$ is a standard basis of $I_p$. We explicitly test 
whether $(f_i \mod p) \in \gen{G \mod p}$ for $i=1,\ldots,r$ and }$(G \mod p) \subseteq 
\texttt{std}(I_p)$\footnote{The procedure \texttt{std} is implemented in \textsc{Singular} and 
computes a Gr\"obner basis resp. standard basis of the input.}. 

\vspace{0.2cm}

This test in positive characteristic accelerates the algorithm enormously. It is much faster 
than in characteristic zero since the standard basis computation in \textsc{pTestSB} is as 
expensive as in any other positive characteristic, i.e., as any other standard basis 
computation within the algorithm. 

If the \textsc{pTestSB} is negative, then $P$ is not sufficiently large for $I$, that is, 
$G$ cannot be a standard basis of $I$ over $\Q$. Contrariwise, if the \textsc{pTestSB} is 
positive, then $G$ is most probably a standard basis of $I$.

Algorithm \ref{algModStd} shows the modular standard basis algorithm.\footnote{The 
corresponding procedures are implemented in \singular in the library \texttt{modstd.lib}.}

\begin{algorithm}[h] 
\caption{\textsc{modStd}} \label{algModStd}
Assume that $>$ is either a global or a local monomial ordering.

\begin{algorithmic}
\REQUIRE $I \subseteq \Q[X]$.
\ENSURE $G \subseteq \Q[X]$ the standard basis of $I$.
\vspace{0.1cm}
\STATE choose $P$, a list of random primes;
\STATE $GP = \emptyset$;
\LOOP
\FOR{$p \in P$}
\STATE compute a standard basis $G_p$ of $I_p$;
\STATE $GP = GP \cup \{G_{p}\}$;
\ENDFOR
\STATE $(GP,P) = \textsc{deleteUnluckyPrimesSB}(GP,P)$;
\STATE lift $(GP,P)$ to $G \subseteq \Q[X]$ by applying Chinese remainder and Farey rational 
       map;
\IF{\textsc{pTestSB}$(I,G,P)$}
\IF{$I \subseteq \gen G$}
\IF{$G$ is a standard basis of $\gen G$}
\RETURN $G$;
\ENDIF
\ENDIF
\ENDIF
\STATE enlarge $P$;
\ENDLOOP
\end{algorithmic}
\end{algorithm}

\begin{rem}
The presented version of the algorithm is just pseudo-code whereas its implementation in 
\textsc{Singular} is optimized. E.g., the standard bases $G_p$ of $I_p \subseteq \F_p[X]$ for 
$p \in P$ are not computed repeatedly, but stored and reused in further iteration steps.
\end{rem}

\begin{rem}
Algorithm \ref{algModStd} can easily be parallelized in the following way:
\begin{enumerate}
\item[(1)] Compute the standard bases $G_p$ in parallel. 
\item[(2)] Parallelize the final tests:
           \begin{itemize}
           \item Check if $I \subseteq \gen G$ by checking if $f \in \gen G$ for all
                 $f \in F_I$.
           \item Check if $G$ is a standard basis of $\gen G$ by checking if every
                 $s$--polynomial not excluded by well-known criteria, vanishes by reduction 
                 w.r.t. $G$. 
           \end{itemize}
\end{enumerate}
\end{rem}

Algorithm \ref{algModStd} terminates by construction, and its correctness is guaranteed by 
the following theorem which is proven in \cite{A03} in the case that  $I$ is homogeneous resp. 
in \cite{Pf07} in the case that the ordering is local. The case that the ordering is global follows
by using weighted homogenization as in Theorem 7.5.1 of \cite{GP07}.

\begin{thm} \label{thmModStd}
Let $G \subseteq \Q[X]$ be a set of polynomials such that $\LM(G)=\LM(G_p)$ where $G_p$ is a 
standard basis of $I_p$ for some prime number $p$, $G$ is a standard basis of $\gen G$ and $I 
\subseteq \gen G$. Then $I = \gen G$.
\end{thm}

Note that the first condition follows from a positive result of \textsc{pTestSB} whereas the 
second and third condition are verified explicitly at the end of the algorithm.

\begin{proof}[Proof of Theorem \ref{thmModStd}]
We assume that $>$ is a global monomial ordering. The proof for a local ordering is similar.
Let $F_I = \{f_1,\ldots,f_r\} \subseteq \Q[X]$ such that $I = \langle F_I \rangle$ and $G = \{g_1,
\ldots,g_s\} \subseteq \Q[X]$. Since $G$ is a standard basis of $\gen G $ w.r.t. $>$ and $I 
\subseteq \gen G$ there exist for each $i = 1,\ldots,r$ polynomials $\xi_{ij} \in \Q[X]$ such that 
$$f_i = \sum_{j=1}^s \xi_{ij} g_j \quad \text{ satisfying }\; \LM_>(f_i) \geq \LM_>(\xi_{ij} g_j) 
\;\text{ for all }\; j = 1,\ldots,s.$$  
Due to Corollary 1.7.9 of \cite{GP07} there exists a finite set $M \subseteq \Mon(X)$ with the
following property: Let $>'$ be any monomial ordering on $\Mon(X)$ coinciding with $>$ on
$M$, then $\LM_>(G) = \LM_{>'}(G)$ and $G$ is also a standard basis of $\gen G$ w.r.t. $>'$. 

Moreover, due to Lemma 1.2.11 resp. Exercise 1.7.17 of \cite{GP07} we possibly enlarge the 
set $M$ and chosse some $w= (w_1,\ldots,w_n) \in \Z_{>0}^n$ such that $> = >_w$ on $M$, 
i.e. $\LM_>(G) = \LM_{>_w}(G)$ resp. $G$ is a standard basis of $\gen G$ w.r.t. $>_w$, 
and\footnote{For a polynomial $f \in \Q[X]$, we define by $\tail(f) := f - \LM(f)$ the \emph{tail} 
of $f$; cf. \cite{GP07}.}
\begin{align*}
\wdeg\left(\LM_{>_w}(f_i)\right) & > \wdeg\left(\LM_{>_w}(\tail(f_i))\right), \\
\wdeg\left(\LM_{>_w}(g_j)\right) & > \wdeg\left(\LM_{>_w}(\tail(g_j))\right), \\
\wdeg\left(\LM_{>_w}(\xi_{ij}g_j)\right) & > \wdeg\left(\LM_{>_w}(\tail(\xi_{ij}g_j))\right), 
\end{align*}
for all $i = 1,\ldots,r$ and $j = 1,\ldots,s$.

Now we consider on $\Q[X,t]$ the weighted degree ordering with weight vector $(w_1,\ldots,
w_n,1)$ refined by $>_w$ on $\Q[X]$ and denote it also by $>_w$. For $f \in \Q[X]$ let $f^h = t^
{\wdeg(f)} \cdot f(x_1/t^{w_1}, \ldots,x_n/t^{w_n})$ be the weighted homogenization of $f$ w.r.t. 
$t$. We set $\overline{F}_I := \set{f_1^h}{f_r^h}$, $\overline I := \gen{\overline F_I}$ and 
$\overline G := \set{g_1^h}{g_s^h}$. Then Proposition 7.5.3 of \cite{GP07} guarantees that
$\overline G \text{ is a standard basis of } \gen{\overline G}$ and since $\LM_{>_w}(G) = 
\LM_{>_w}(G_p)$ it also holds by construction that $\LM_{>_w}(\overline G) = \LM_{>_w}
(\overline G_p)$. Now let $i \in \set 1r$, then $f_i = \sum_{j=1}^s \xi_{ij} g_j$ satisfying 
$\LM_{>_w}(f_i) \geq_w \LM_{>_w}(\xi_{ij} g_j)$ for all $j = 1,\ldots,s$. This implies $\wdeg(f_i)
\geq \wdeg(\xi_{ij}g_j)$ for all $j = 1,\ldots,s$ by the choice of $w \in \Z_{>0}^n$. Consequently
we have 
\begin{align*}
t^{\wdeg(f_i)} & f\left(\frac{x_1}{t^{w_1}},\ldots,\frac{x_n}{t^{w_n}}\right) \\ &= 
\sum_{j=1}^s t^{\wdeg(f_i)} \xi_{ij}\left(\frac{x_1}{t^{w_1}},\ldots,\frac{x_n}{t^{w_n}}\right)
g_j\left(\frac{x_1}{t^{w_1}},\ldots,\frac{x_n}{t^{w_n}}\right) \in \gen{\overline G},
\end{align*}
thus $f_i^h \in \gen{\overline G}$ resp. $\overline I \subseteq \gen{\overline G}$ since $i \in \set 1r$
was arbitrarily chosen.

It remains to prove that $\overline I = \gen{\overline G}$. Let $n \in \N$. We know that $\overline 
I_p = \gen{\overline G_p}$ due to the fact that $\LM_{>_w}(\overline G) = \LM_{>_w}
(\overline G_p)$, so especially it holds $\HF_{\overline I_p}(n) = \HF_{\gen{\overline G_p}}(n)
= \HF_{\gen{\overline G}}(n)$ for the corresponding Hilbert functions. On the other hand we have 
$$\HF_{\overline I}(n) \leq \HF_{\overline I_p} = \HF_{\gen{\overline G}}(n) \leq \HF_{\overline I}(n) 
< \infty,$$ where the second inequality is true since $\overline I \subseteq \gen{\overline G}$. The 
first inequality follows from the fact that $\dim_\Q (\overline I[n]) \geq \dim_{\F_p} 
(\overline I_p[n])$, where $\overline I[n]$ resp. $\overline I_p[n]$ denotes the vector space 
generated by all (weighted) homogeneous polynomials of degree $n$.  Namely we can find a 
$\Q$-basis of $\overline I[n]$ of polynomials in $\Z[X,t] \cap \overline I$ which induces generators 
of $\overline I_p[n]$.
\end{proof}

\begin{rem} \label{remModStdVerification}
Algorithm \ref{algModStd} is also applicable without applying the final tests, i.e. skipping the
verification that $I \subseteq \gen G$ and $G$ is a standard basis of $\gen G$. In this case the 
algorithm is probabilistic, i.e. the output $G$ is a standard basis of the input $I$ only with high 
probability. This usually accelerates the algorithm enormously.  
Note that the probabilistic algorithm works for any ordering, i.e. also for the so--called mixed 
ordering. In case of a mixed ordering one could homogenize the ideal $I$, compute a standard
basis using \textsc{modStd} and dehomogenize afterwards. Experiments showed that this is 
usually not efficient since the standard basis of the homogenized input has often much more 
elements than the standard basis of the ideal we started with.
\end{rem}

\section{A modular approach to primary decomposition} \label{secAssPrimes}

In the following let $I\subseteq \mathbb{Q}[X]$ be a zero--dimensional ideal and $d:= \dim_\Q
(\Q[X]/I)$. Within this section we describe an algorithm for computing the associated primes of
$I$ using modular methods. In conclusion we make remarks how to achieve the corresponding
primary ideals from the associated primes of $I$.

The following well--known proposition (cf. \cite{GTZ88} or \cite{GP07}) describes how to 
compute the associated prime ideals of a radical ideal over $\Q$. Note that these results are 
also valid for perfect infinite fields.

\begin{prop} \label{propCoordinateChange}
Let $I\subseteq \Q[X]$ be a radical ideal.
\begin{enumerate}
\item Let $\gen F = I\cap \Q[x_n]$ and assume $\deg(F)=\dim_\Q(\Q[X]/I)$. Let $F=
          F_1\cdots F_s$ be the factorization of $F$ into irreducible factors over $\Q$. 
          Then $I=\bigcap_{i=1}^s \langle I, F_i\rangle$ and $\langle I, F_i\rangle$ is prime for 
          $i = 1,\ldots,s$.
\item There exists a non--empty Zariski open subset $U\subseteq \Q^{n-1}$ such that for all 
          $a=(a_1, \ldots, a_{n-1})\in U$ the linear coordinate change $\varphi_a$ defined by 
          $\varphi_a(x_i)=x_i$ for $i=1,\ldots,n-1$ and $\varphi_a(x_n)=x_n+\sum_{i=1}^{n-1} a_ix_i$ 
          satisfies $$\dim_\Q \big(\Q[X]/\varphi_a(I)\big)=\dim _\Q \big(\Q[x_n]/(\varphi_a(I)\cap 
          \Q[x_n])\big).$$
\end{enumerate}
\end{prop}

\begin{cor} \label{corPDofRad}
Let $F\in \Q[T]$, $T$ a variable, be squarefree and $r=x_n+\sum_{i=1}^{n-1} a_ix_i$ with
$a_1,\ldots,a_{n-1} \in \Z$ such that $\deg(F)=\dim_\Q(\Q[X]/I)$, and $F(r)\in I$ but no proper factor
of $F(r)$ is in $I$, then $I$ is a radical ideal. 
Let $F=F_1\cdots F_s$ be the factorization of $F$ into irreducible factors over $\Q$. 
Then $I=\bigcap_{i=1}^s \langle I, F_i(r)\rangle$ and $\langle I, F_i(r)\rangle$ is prime 
for $i = 1,\ldots,s$.
\end{cor}

\begin{proof}
Using a linear change of variables we may assume that $r = x_n$. Since no proper factor of $F(r)$
is in $I$ we obtain $\langle F(x_n) \rangle = I \cap \Q[x_n]$. Since $\deg(F) = \dim_\Q(\Q[X]/I)$ we 
have $I = \gen{x_1-h_1(x_n), \ldots, x_{n-1}-h_{n-1}(x_n),F(x_n)}$ for suitable $h_1,\ldots,h_{n-1} 
\in \Q[x_n]$. Thus, $I$ is radical because $F$ is squarefree. The rest is an immediate consequence
of Proposition \ref{propCoordinateChange}(1).
\end{proof}

Consequently, for the computation of the primary decomposition, we firstly verify whether $I$ is 
already radical. Therefore we choose a generic linear form $r = a_1x_1+\ldots+a_{n-1}x_{n-1}+x_n$ 
with $a_1,\ldots,a_{n-1} \in \Z$, and use a test in positive characteristic, similarly to section 
\ref{secModStd}.

\vspace{0.2cm}

\emph{\textsc{pTestRad:} We randomly choose a prime number $p$ such that $\dim_{\F_p}
(\F_p[X]/I_p) = d$. Let $\varphi:\F_p[T] \longrightarrow \F_p[X]$ be defined by $\varphi(T) = r 
\mod p$ (cf. Lemma \ref{lemF}(1)) and $\gen{F_p} := \varphi^{-1}(I_p)$. We test if $\deg(F_p) 
= d$.} 

\vspace{0.2cm}

In case of a negative result of the test there is a high probability that the ideal is not radical (cf. 
Proposition \ref{propCoordinateChange}(2)) and we compute the radical using modular methods. 
The computation of the radical is usually much more time consuming than the \textsc{pTestRad} 
even if the ideal is already radical. 
The following proposition (cf. \cite{KrLo91}, \cite{GP07}) is the basis for computing the 
radical of a zero--dimensional ideal.

\begin{prop} \label{propRad}
Let $I \subseteq \Q[X]$ be a zero--dimensional ideal and $\gen{f_i} = I \cap \Q[x_i]$ for $i = 
1,\ldots,n$. Moreover, let $g_i$ be the squarefree part of $f_i$. Then the following holds.
\begin{enumerate}
\item $\sqrt I = I + \gen{g_1,\ldots,g_n}$. 
\item If $\deg(f_n) = \dim_\Q (\Q[X]/I)$ then $\sqrt I = \gen{I,g_n}$.
\end{enumerate}
\end{prop}

\begin{proof}
Part (1) of the proposition is proved in \cite{KrLo91}. For part (2) we notice that if it holds $\deg(f_n) 
= \dim_\Q(\Q[X]/I)$ then there exist $h_1,\ldots,h_{n-1} \in \Q[x_n]$ such that $\{x_1-h_1,\ldots,
x_{n-1}-h_{n-1}, f_n\}$ is a Gr\"obner basis of $I$ w.r.t. the lexicographical ordering $x_1 > \ldots > 
x_n$. Thus, we have $\sqrt I = \gen{x_1-h_1, \ldots,x_{n-1}-h_{n-1},g_n}$.
\end{proof}

With analogous considerations as in section \ref{secModStd}, the essential idea of the algorithm to
compute the radical of $I$ is as follows. Choose a set $P$ of prime numbers, compute, for every 
$p \in P$, monic polynomials $f_1^{(p)},\ldots,f_n^{(p)}$ satisfying $\langle f_i^{(p)} \rangle = I_p \cap 
\F_p[x_i]$ for $i = 1,\ldots,n$ and finally lift these polynomials via Chinese remainder algorithm and
Farey rational map to $(f_1,\ldots,f_n) \in \Q[x_1] \times \ldots \times \Q[x_n]$.

\begin{defn}
Let  $(f_1,\ldots,f_n) \in \Q[x_1] \times \ldots \times \Q[x_n]$ satisfy $\gen{f_i} = I \cap \Q[x_i]$ for
$i = 1,\ldots,n$.\footnote{By abuse of notation we use the same terminology as in Definition 
\ref{defnLucky} since it is always clear out of context which definition we are referring to.}
\begin{enumerate}
\item If $(f_1^{(p)},\ldots,f_n^{(p)}) \in \F_p[x_1] \times \ldots \times \F_p[x_n]$ satisfies $\langle f_i^{(p)}
          \rangle = I_p \cap \F_p[x_i]$ for $i = 1,\ldots,n$, then the prime number $p$ is called \emph{lucky 
          for $I$} if and only if $\deg(f_i) = \deg(f_i^{(p)})$ for $i = 1,\ldots,n$. Otherwise $p$ is called 
          \emph{unlucky for $I$}.
\item A set $P$ of lucky primes for $I$ is called \emph{sufficiently large for $I$} if and only if 
          $\prod_{p \in P} p \geq \max\{2 \cdot |c|^2 \mid c \,\text{ coefficient occuring in }\, f_1,\ldots,f_n\}$. 
\end{enumerate}
\end{defn}

After having computed the set $FP:=\{(f_1^{(p)},\ldots,f_n^{(p)}) \mid p \in P\}$ we delete the unlucky 
primes in the following way.

\vspace{0.2cm}

\emph{\textsc{deleteUnluckyPrimesRad:} We define an equivalence relation on 
$(FP,P)$ by \linebreak $(F^{(p)},p) \sim (F^{(q)},q) :\Longleftrightarrow \deg(f_i^{(p)}) = \deg(f_i^{(q)})$
for $i=1,\ldots,n$. Then the equivalence class of largest cardinality is stored in $(FP,P)$, the others 
are deleted.} 

\vspace{0.2cm}

With the aid of this method we are able to choose a set of lucky primes with high probability. 
A faulty decision will be compensated by the subsequent test whether $f_i \in I$ for $i = 1,\ldots,n$.

Since we cannot predict if a given set of primes $P$ is sufficiently large for $I$, we have to proceed
by trial and error as already described in section \ref{secModStd}.

Algorithm \ref{algZeroRadical} computes the radical of $I$.\footnote{The 
corresponding procedure is implemented in \singular in the library \texttt{assprimeszerodim.lib}.}

\begin{algorithm}[h]
\caption{\textsc{zeroRadical}} \label{algZeroRadical}
\begin{algorithmic}
\REQUIRE $I = \gen{G_I} \subseteq \Q[X]$ a zero--dimensional ideal generated by a Gr\"obner basis
                     $G_I$ w.r.t. some global ordering.
\ENSURE $G \subseteq \Q[X]$ a Gr\"obner basis of the radical of $I$ w.r.t. a degree--ordering.
\vspace{0.1cm}
\STATE choose $P$, a list of random primes;
\STATE $FP = \emptyset$;
\LOOP
\FOR{$p \in P$}
\STATE compute monic polynomials $f_i^{(p)}$ such that $\gen{f_i^{(p)}} = I_p \cap \F_p[x_i]$ for $i=1,\ldots,n$;
\STATE $FP = FP \cup \{ (f_1^{(p)},\ldots,f_n^{(p)}) \}$;
\ENDFOR
\STATE $(FP,P) = \textsc{deleteUnluckyPrimesRad}(FP,P)$;
\STATE lift $(FP,P)$ to $(f_1,\ldots,f_n) \in \Q[x_1] \times \ldots \times \Q[x_n]$ by applying Chinese remainder 
               and Farey rational map;
\STATE use $G_I$ to test if $f_i \in I$ for $i=1,\ldots,n$;
\IF{$f_i \in I$ for all $i = 1,\ldots,n$}
\STATE exit loop;
\ENDIF
\STATE enlarge $P$;
\ENDLOOP
\FOR{$i = 1,\ldots,n$}
\STATE compute $g_i$, the squarefree part of $f_i$;
\ENDFOR
\STATE $I = I + \gen{g_1,\ldots,g_n}$;
\STATE compute $G\subseteq \Z[X]$, a $\Q[X]$--Gr\"obner basis of $I$ w.r.t. a degree--ordering;\footnotemark
\RETURN $G$;
\end{algorithmic}
\end{algorithm}

If the \textsc{pTestRad} is positive then, with high probability, after a generic coordinate change it holds 
$\dim_\Q(\Q[x_n]/(I \cap \Q[x_n])) = d$. In this case it is not necessary to compute the radical of $I$ and we
rely on the following corollary.

\pagebreak
\footnotetext[12]{Here we use the procedure \textsc{modStd} as described in section \ref{secModStd}.}

\begin{cor} \label{corRad}
Let $I \subseteq \Q[X]$ be a zero--dimensional ideal and $r = x_n + \sum_{i=1}^{n-1} a_ix_i$
with $a_1,\ldots,a_{n-1} \in \Z$. Let $F \in \Q[T]$, $T$ be a variable, such that $\deg(F) = \dim_\Q(\Q[X]/I)$ 
and $F(r) \in I$ but no proper factor of $F(r)$ is in $I$. Moreover, let $H$ be the squarefree part of $F$. 
Then $\sqrt I = \gen{I,H(r)}$.
\end{cor}

\begin{proof}
The proof is a consequence of Proposition \ref{propRad}(2) and Corollary \ref{corPDofRad}.
\end{proof}

Consequently we need to obtain a polynomial $F \in \Q[T]$ satisfying the required properties of 
Corollary \ref{corPDofRad} resp. Corollary \ref{corRad}. The following lemma is helpful in this
direction.

\begin{lem} \label{lemF}
Let $K$ be a field\footnote{We substitute $\Q$ by an arbitrary field $K$ since we also need the results 
of Lemma \ref{lemF} for finite fields.}, $F\in K[T]$, $T$ a variable, be monic and squarefree, let $r=x_n+
\sum_{i=1}^{n-1} a_ix_i$, $a_1,\ldots,a_{n-1} \in K$, such that $\deg(F)=\dim_K (K[X]/I)$ and $F(r)
\in I$ but no proper factor of $F(r)$ is in $I$.
\begin{enumerate}
\item Let $\varphi: K[T]\to K[X]$ be defined by $\varphi(T)=r$. Then $\varphi^{-1}(I)=\langle 
          F\rangle$.
\item Let $\psi:K[X]\to K[X]$ be defined by $\psi(x_i)=x_i$ for $i=1,\ldots,n-1$ and $\psi(x_n)=2x_n-r$. 
          Then $\psi (I)\cap K[x_n]=\langle F(x_n)\rangle$.
\item Let $\lambda:K[X]/I\to K[X]/I$ be the map defined by the multiplication with $r$, 
          $\lambda(g+I)=r\cdot g+I$. Then $F$ is the characteristic polynomial of $\lambda$.
\end{enumerate}
\end{lem}

\begin{proof}
\begin{enumerate}
\item Since $\varphi(F) = F(r) \in I$ we obtain $F \in \varphi^{-1}(I)$. Thus we have $\gen F = \varphi^{-1}(I)$
          because no proper factor of $F(r)$ is in $I$.
\item It holds $F(x_n) = \psi(F(r)) \in \psi(I)$ by definition of $\psi$. The assumption implies that no proper 
          factor of $F(x_n)$ is in $\psi(I)$, i.e. $\gen{F(x_n)} = \psi(I) \cap K[x_n]$.
\item Using the map $\psi$ of (2) we may assume $r = x_n$. As in the proof of Corollary \ref{corPDofRad}
          we obtain $I = \gen{x_1-h_1,\ldots,x_{n-1}-h_{n-1}, F(x_n)}$ for suitable $h_1,\ldots,h_{n-1} \in 
          K[x_n]$ since $\deg(F) = \dim_K(K[X]/I) = d$. Hence, we may choose $\{1,x_n,\ldots,x_n^{d-1}\}$ as 
          a basis of $K[X]/I \cong K[x_n]/\gen{F(x_n)}$, and obtain the polynomial $F$ to be the characteristic 
          polynomial of the multiplication with $x_n$.
\end{enumerate}
\end{proof}

Lemma \ref{lemF} shows that the approach of Eisenbud, Hunecke, Vasconcelos (cf. 
\cite{EHV92}) using (1) of the lemma, the approach of Gianni, Trager, Zacharias (cf. 
\cite{GTZ88}) using (2) of the lemma and the approach of Monico (cf. \cite{M02}) using (3) 
of the remark are in principle the same. The computations for (1) resp. (2) require Gr\"obner 
bases with respect to suitable block--orderings whereas in (3) we do not need a special 
ordering for the Gr\"obner basis but we have to compute a determinant. All three algorithms 
are implemented in \textsc{Singular}. 

\begin{rem}
We can also compute the polynomial $F \in \Q[T]$ using modular methods. For this purpose
we compute $F^{(p)} \in \F_p[T]$ monic such that $\gen{F^{(p)}} = \ker(\varphi_p)$, whereat
$\varphi_p: \F_p[T] \longrightarrow \F_p[X]/I_p, \, \varphi_p(T) = r \mod I_p$, for several prime
numbers $p$ and preserve just those $F^{(p)}$ with $\deg(F^{(p)}) = d$. Afterwards we lift the
results to $F \in \Q[T]$ by applying Chinese remainder and Farey rational map. 
\end{rem}

\begin{rem}
If $K = \C$ is the field of complex numbers we can use the polynomial $F$ of Corollary 
\ref{corPDofRad} to compute the zeros of the ideal $I$. The zeros of $F$ are the eigenvalues 
of the multiplication map $\lambda$ defined in Lemma \ref{lemF}. Let $\lambda_1,\ldots,\lambda_d$ 
be the (different) eigenvalues of $\lambda$ then $I = \bigcap_{i=1}^d \gen{I,r-\lambda_i}$. 
Moreover, $\gen{I,r-\lambda_i}$ is a maximal ideal in $\C[X]$ representing a zero of $I$ for $i = 1,
\ldots,d$.
\end{rem}

Referring to Proposition \ref{propCoordinateChange}, Corollary \ref{corPDofRad} and the above
considerations, Algorithm \ref{algAssPrimes} computes the associated primes of $I$.\footnote{The 
corresponding procedures are implemented in \singular in the library \texttt{assprimeszerodim.lib}.}

\begin{algorithm}[h] 
\caption{\textsc{assPrimes}} \label{algAssPrimes}
\begin{algorithmic}
\REQUIRE $I\subseteq \Q[X]$ a zero--dimensional ideal.
\ENSURE  $L=\{M_1, \ldots, M_s\}$, $M_i$ prime and $\sqrt I=\bigcap_{i=1}^s M_i$.
\vspace{0.1cm}
\STATE compute $G\subseteq \Z[X]$, a $\Q[X]$--Gr\"obner basis of $I$ w.r.t. a 
               degree--ordering;\footnotemark
\STATE compute $d=\dim_{{\mathbb Q}}({\mathbb Q}[X]/I)$ using $G$;
\STATE choose $a_1,\ldots,a_{n-1} \in \Z$ randomly, $r=a_1x_1+\ldots +a_{n-1}x_{n-1}+x_n$;
\IF{not \textsc{pTestRad}$(d,r,G)$}
\STATE $G = \textsc{zeroRadical}(G)$;
\STATE $d = \dim_\Q(\Q[X]/\gen G)$;
\ENDIF
\STATE choose $P$, a list of random primes;
\STATE $FP = \emptyset$;
\STATE $l = 0$;
\LOOP
\FOR{$p \in P$}
\STATE compute $F^{(p)} \in \F_p[T]$ monic such that $\gen{F^{(p)}} = \ker(\varphi_p)$, whereat
               $\varphi_p: \F_p[T] \longrightarrow \F_p[X]/I_p, \, \varphi_p(T) = r \mod I_p$;\footnotemark
\IF{$\deg(F^{(p)}) = d$}
\STATE $FP = FP \cup \{ F^{(p)} \}$;
\ENDIF
\ENDFOR
\IF{$\#(FP)=l$}
\STATE $G = \textsc{zeroRadical}(G)$;
\STATE $d = \dim_\Q(\Q[X]/\gen G)$;
\STATE choose $a_1,\ldots,a_{n-1} \in \Z$ randomly, $r=a_1x_1+\ldots +a_{n-1}x_{n-1}+x_n$;
\ELSE
\STATE lift $(FP,P)$ to $F \in \Q[T]$ by applying Chinese remainder and Farey rational map;
\STATE factorize $F=F_1^{\nu_1}\cdots F_s^{\nu_s}$ with $F_1,\ldots,F_s$ irreducible;
\STATE compute $F(r)$ and $F_1(r),\ldots,F_s(r)$;
\IF{$F(r)\in I$}
\IF{no proper factor of $F(r)$ is in $I$}
\RETURN $\{\langle I, F_1(r)\rangle, \ldots, \langle I, F_s(r)\rangle\}$;
\ELSE
\STATE choose a non--trivial factor $H$ of $F$ of minimal degree such that $H(r) \in I$;
\STATE let $F_{i_1},\ldots,F_{i_t}$ correspond to $H$;
\RETURN $\textsc{assPrimes}(\gen{I,F_{i_1}(r)}) \cup \ldots \cup \textsc{assPrimes}(\gen{I,F_{i_t}(r)})$;
\ENDIF
\ENDIF
\STATE enlarge $P$;
\STATE $l = \#(FP)$;
\ENDIF
\ENDLOOP
\end{algorithmic}
\end{algorithm}

\begin{rem}
The presented versions of Algorithms \ref{algZeroRadical} and \ref{algAssPrimes} are just pseudo-code 
whereas their implementation in \textsc{Singular} is optimized. E.g., the polynomials $f_i^{(p)} \in \F_p[x_i]$ 
resp. $F^{(p)} \in \F_p[T]$  for $p \in P$ are not computed repeatedly, but stored and reused in further 
iteration steps.
\end{rem}

\begin{rem} \label{remAssPrimes}
Algorithm \ref{algZeroRadical} resp. Algorithm \ref{algAssPrimes} can easily be parallelized by computing 
the polynomials $f_i^{(p)} \in \F_p[x_i]$ resp. $F^{(p)} \in \F_p[T]$ in parallel. Experiments indicate that the 
difficult and time consuming part of Algorithm \ref{algAssPrimes} is the test whether $F(r)\in I$ and the 
computation of $F_1(r),\ldots,F_s(r)$. These $s+1$ computations are independent from each other such 
that they can also be verified separately in parallel.

Following the idea of one of the referees we tried to avoid the computation of $F(r)$ by computing a
$\Q[X,T]$--Gr\"obner basis of $\gen{I,T-r}$ w.r.t. an elimination ordering (eliminating $X$) by using modular
methods (cf. section \ref{secModStd}) and FGLM--algorithm (cf. \cite{FGLM93}). In this case we directly 
compute $\gen{I,T-r}_{\Q[X,T]} \cap \Q[T] = \gen F$ and may consequently omit the verification. Experiments 
showed that this is as time consuming as the presented method in Algorithm \ref{algAssPrimes}.
\end{rem}

\footnotetext[15]{Here we use the procedure \textsc{modStd} as described in section \ref{secModStd}.}
\footnotetext[16]{All approaches mentioned in Lemma \ref{lemF} are applicable to verify this step.}

\begin{rem}
Knowing the associated primes it is easy to compute the primary ideals using the method of Shimoyama 
and Yokoyama (cf. \cite{SY96}): Let $M_1,\ldots,M_s$ be the associated primes of the zero--dimensional 
ideal $I$ and $\sigma_1,\ldots,\sigma_s$ a system of separators, i.e. $\sigma_i \notin M_i$ and $\sigma_i 
\in M_j$ for $j \neq i$, then the saturation of $I$ w.r.t. $\sigma_i$ is the primary ideal corresponding to $M_i$. 
Each $\sigma_i$ can be chosen as $\prod_{j \neq i} m_j$ whereat $m_j$ is an element of a Gr\"obner basis 
of $M_j$ which is not in $M_i$. The saturation can be computed modularly, similarly to \textsc{modStd} and 
in parallel.
\end{rem}

\section{Examples, timings and conclusion} \label{secExTime}

In this section we provide examples on which we time the algorithms \texttt{modStd} (cf. 
section \ref{secModStd}) resp. \texttt{assPrimes} (cf. section \ref{secAssPrimes}) and
their parallelizations as opposed to the usual algorithms \texttt{std} resp. 
\texttt{minAssGTZ}\footnote{The procedure \texttt{minAssGTZ} is implemented in \singular 
in the library \texttt{primdec.lib} and computes the minimal associated prime ideals of the 
input.} implemented in \textsc{Singular}. Timings are conducted by using the 32-bit version 
of \singular{3-1-2} on an AMD Opteron 6174 with $48$ CPUs, 800 MHz each, 128 GB RAM 
under the Gentoo Linux operating system. All examples are chosen from The SymbolicData 
Project (cf. \cite{G10}).

\begin{rem}
The parallelization of our modular algorithms is attained via multiple processes organized 
by \singular library code. Consequently a future aim is to enable parallelization in the kernel 
via multiple threads.
\end{rem}

We choose the following examples to emphasize the superiority of modular standard basis 
computation and especially its parallelization:

\begin{exmp} \label{ex1}
Characteristic: $0$, ordering: \texttt{dp}\footnote{\emph{Degree reverse lexicographical ordering:} 
Let $X^\alpha, X^\beta \in \Mon(X)$. $X^\alpha >_{dp} X^\beta \, :\Longleftrightarrow \, \deg(X^\alpha) 
> \deg(X^\beta)$ or $(\deg(X^\alpha) = \deg(X^\beta)$ and $\exists \, 1 \leq i \leq n: \; \alpha_n = \beta_n, 
\ldots, \alpha_{i-1} = \beta_{i-1}, \alpha_i < \beta_i)$, where $\deg(X^\alpha) = \alpha_1 + \ldots + \alpha_n$; 
cf. \cite{GP07}.},  \texttt{Cyclic\_8.xml} (cf. \cite{BF91}).
\end{exmp}

\begin{exmp} \label{ex2}
Characteristic: $0$, ordering: \texttt{dp}, \texttt{Paris.ilias13.xml} (cf. \cite{KoLa99}).
\end{exmp}

\begin{exmp} \label{ex3}
Characteristic: $0$, ordering: \texttt{dp}, homog. \texttt{Cyclic\_7.xml} (cf. 
\cite{BF91}).
\end{exmp}

\begin{exmp} \label{ex4}
Characteristic: $0$, ordering: \texttt{ds}\footnote{\emph{Negative degree reverse lexicographical ordering:} 
Let $X^\alpha, X^\beta \in \Mon(X)$. $X^\alpha >_{ds} X^\beta \, :\Longleftrightarrow \, \deg(X^\alpha) 
< \deg(X^\beta)$ or $(\deg(X^\alpha) = \deg(X^\beta)$ and $\exists \, 1 \leq i \leq n: \; \alpha_n = \beta_n, 
\ldots, \alpha_{i-1} = \beta_{i-1}, \alpha_i < \beta_i)$, where $\deg(X^\alpha) = \alpha_1 + \ldots + \alpha_n$; 
cf. \cite{GP07}.}, \texttt{Steidel\_1.xml} (cf. \cite{Pf07}).
\end{exmp}

Table \ref{tabModStd1} summarizes the results where $\texttt{modStd}^*(n)$ denotes the parallelized 
version of the algorithm applied on $n$ cores. In all tables, the symbol "-" indicates out of memory 
failures. All timings are given in seconds.

\begin{table}[hbt]
\begin{center}
\begin{tabular}{|r|r|r|r|r|r|}
\hline
Exmp. & \multicolumn{1}{c|}{\texttt{std}} & \multicolumn{1}{c|}{\texttt{modStd}} 
& \multicolumn{1}{c|}{$\texttt{modStd}^*(4)$} & \multicolumn{1}{c|}{$\texttt{modStd}^*(9)$}
& \multicolumn{1}{c|}{$\texttt{modStd}^*(30)$} \\
\hline \hline
\ref{ex1} & - & 8271 & 4120 & 2927 & 1138 \\ \hline
\ref{ex2} & 37734 & 1159 & 676 & 580 & 380 \\ \hline
\ref{ex3} & 3343 & 3436 & 886 & 408 & 113 \\ \hline
\ref{ex4} & \hspace{1.4cm} - & \hspace{1.4cm} 6 & \hspace{1.4cm} 3 & \hspace{1.4cm} 3 
                & \hspace{1.4cm} 3 \\ \hline 
\end{tabular}
\end{center}
\hspace{15mm}
\caption{Total running times for computing a standard basis of the considered examples via 
\texttt{std}, \texttt{modStd} and its parallelized variant $\texttt{modStd}^*(n)$ for $n = 4,9,30$.} 
\label{tabModStd1}
\end{table}

The basic algorithm \texttt{std} runs out of memory for examples \ref{ex1} and \ref{ex4}.
As mentioned in section \ref{secModStd}, it is possible to parallelize the computation in 
several parts of the algorithm \texttt{modStd}. In many cases it turns out that the final 
test - the verification whether the lifted set of polynomials includes the input and is 
itself a standard basis, see also Remark \ref{remModStdVerification} - is a time consuming 
part. Therefore we extract the timings for the computation without the verification test in Table 
\ref{tabModStd2}, again in seconds.

\begin{table}[hbt]
\begin{center}
\begin{tabular}{|r|r|r|r|r|}
\hline
Exmp. & \multicolumn{1}{c|}{$\texttt{modStd}_{\text{w/o v.}}$} 
& \multicolumn{1}{c|}{$\texttt{modStd}^*_{\text{w/o v.}}(4)$} 
& \multicolumn{1}{c|}{$\texttt{modStd}^*_{\text{w/o v.}}(9)$} 
& \multicolumn{1}{c|}{$\texttt{modStd}^*_{\text{w/o v.}}(30)$} \\
\hline \hline
\ref{ex1} & 7929 & 3751 & 2698 & 920 \\ \hline
\ref{ex2} & 941 & 614 & 552 & 370 \\ \hline
\ref{ex3} & 52 & 38 & 31 & 36 \\ \hline
\ref{ex4} & \hspace{2.1cm} 6 & \hspace{2.1cm} 3 & \hspace{2.1cm} 3 & \hspace{2.1cm} 3 \\ \hline
\end{tabular}
\end{center}
\hspace{15mm}
\caption{Running times for  \texttt{modStd} and $\texttt{modStd}^*(n)$ with $n = 4,9,30$ without 
verification test.} 
\label{tabModStd2}
\end{table}

We consider the following examples for the computation of the associated prime ideals of a 
given zero--dimensional ideal :

\begin{exmp} \label{ex5}
Characteristic: $0$, ordering: \texttt{dp}, \texttt{Becker-Niermann.xml} (cf. \cite{DGP98}).
\end{exmp}

\begin{exmp} \label{ex6}
Characteristic: $0$, ordering: \texttt{dp}, \texttt{FourBodyProblem.xml} (cf. \cite{BM10}).
\end{exmp}

\begin{exmp} \label{ex7}
Characteristic: $0$, ordering: \texttt{dp}, \texttt{Reimer\_5.xml} (cf. \cite{BM10}).
\end{exmp}

\begin{exmp} \label{ex8}
Characteristic: $0$, ordering: \texttt{lp}\footnote{\emph{Lexicographical ordering:} 
Let $X^\alpha, X^\beta \in \Mon(X)$. $X^\alpha >_{lp} X^\beta \, :\Longleftrightarrow \, \exists \, 1 \leq i \leq n: 
\; \alpha_1 = \beta_1, 
\ldots, \alpha_{i-1} = \beta_{i-1}, \alpha_i > \beta_i$; 
cf. \cite{GP07}.}, \texttt{ZeroDim.example\_12.xml} (cf. \cite{G10}).
\end{exmp}

\begin{exmp} \label{ex9}
Characteristic: $0$, ordering: \texttt{dp}, \texttt{Cassou\_1.xml} (cf. \cite{BM10}).
\end{exmp}

Using modular methods via the algorithm \texttt{assPrimes} we apply all three variants 
mentioned in section \ref{secAssPrimes}. 
\vspace{0.05cm}
\begin{enumerate}
\item[(1)] approach of Eisenbud, Hunecke, Vasconcelos (cf. \cite{EHV92}),
\item[(2)] approach of Gianni, Trager, Zacharias (cf. \cite{GTZ88}),
\item[(3)] approach of Monico (cf. \cite{M02}).
\end{enumerate}
\vspace{0.05cm}
We summarize the results of the timings in Table \ref{tabAssPrimes1} and \ref{tabAssPrimes2} 
where $\texttt{assPrimes}^*(n)$ denotes the parallelized version of the algorithm applied on $n$
cores.

\begin{table}[h]
\begin{center}
\begin{tabular}{|r|r|r|r|r|r|r|r|r|r|r|}
\hline
Exmp. & \texttt{minAssGTZ} & \multicolumn{3}{|c|}{\texttt{assPrimes}} & 
\multicolumn{3}{|c|}{$\texttt{assPrimes}^*(4)$} &
\multicolumn{3}{|c|}{$\texttt{assPrimes}^*(9)$} \\
 & & \;\;(1) & \;\;(2) & \;\;(3) & \;\;(1) & \;\;(2) & \;\;(3) & \;\;(1) & \;\;(2) & \;\;(3) \\
\hline \hline
\ref{ex5} & - & 1 & 1 & 0 & 1 & 1 & 1 & 1 & 1 & 1 \\ \hline
\ref{ex6} & - & 169 & 169 & 188 & 104 & 98 & 104 & 95 & 100 & 105 \\ \hline
\ref{ex7} & - & 129 & 131 & 230 & 90 & 87 & 114 & 76 & 77 & 103 \\ \hline
\ref{ex8} & 189 & 4 & 5 & 5 & 10 & 8 & 8 & 8 & 8 & 8 \\ \hline
\ref{ex9} & 589 & 35 & 35 & 35 & 24 & 23 & 19 & 25 & 24 & 25 \\ \hline
\end{tabular}
\end{center}
\hspace{15mm}
\caption{Total running times for computing the associated prime ideals of the considered 
examples via \texttt{minAssGTZ}, \texttt{assPrimes} and its parallelized variant 
$\texttt{assPrimes}^*(n)$ for $n = 4,9$.} \label{tabAssPrimes1}
\end{table}

The usual algorithm \texttt{minAssGTZ} runs out of memory for examples \ref{ex5}, \ref{ex6}
and \ref{ex7}. Analogously to the modular standard basis algorithm, we also list the timings 
needed for \texttt{assPrimes} resp. $\texttt{assPrimes}^*(n)$ without the final verification step - 
the check whether $F(r) \in I$ and the computation of $F_1(r), \ldots, F_s(r)$, see also Remark 
\ref{remAssPrimes} - in Table \ref{tabAssPrimes2}.

\begin{table}[h]
\begin{center}
\begin{tabular}{|r|r|r|r|r|r|r|r|r|r|}
\hline
Exmp. & \multicolumn{3}{|c|}{$\texttt{assPrimes}_{\text{w/o ver.}}$} & 
\multicolumn{3}{|c|}{$\texttt{assPrimes}^*_{\text{w/o ver.}}(4)$} &
\multicolumn{3}{|c|}{$\texttt{assPrimes}^*_{\text{w/o ver.}}(9)$} \\
 & \;\quad(1) & \;\quad(2) & \;\quad(3) 
 & \;\quad(1) & \;\quad(2) & \:\quad(3) 
 & \;\quad(1) & \;\quad(2) & \;\quad(3) \\
\hline \hline
\ref{ex5} & 1 & 1 & 0 & 1 & 0 & 0 & 1 & 1 & 1 \\ \hline
\ref{ex6} & 15 & 14 & 34 & 7 & 7 & 13 & 5 & 5 & 15 \\ \hline
\ref{ex7} & 41 & 37 & 139 & 39 & 38 & 64 & 30 & 26 & 55 \\ \hline
\ref{ex8} & 4 & 5 & 5 & 9 & 8 & 8 & 8 & 8 & 8 \\ \hline 
\ref{ex9} & 7 & 6 & 7 & 5 & 5 & 5 & 5 & 4 & 6 \\ \hline
\end{tabular}
\end{center}
\hspace{15mm}
\caption{Running times for \texttt{assPrimes} and $\texttt{assPrimes}^*(n)$ with $n = 4,9$ without 
final verification step.} \label{tabAssPrimes2}
\end{table}

\begin{con} \
\vspace{0.15cm}
\begin{enumerate}
\item For the computation of Gr\"obner bases resp. standard bases of ideals $I \subseteq \Q[X]$
          w.r.t. global resp. local orderings \textsc{modStd} should be used. This is usually faster even
          without parallel computing.
\item The probabilistic algorithm to compute standard bases works without any restriction to the 
          ordering. It is much faster than the deterministic one. It can be used to obtain ideas in
          Algebraic Geometry and other fields by computing several examples, similarly to computations
          in positive characteristic $20$ years ago when computations of standard bases in characteristic
          zero have been impossible resp. too slow.  
\item A kernel--implementation of \textsc{modStd} could speed up the modular part using the 
          trace--algorithm of Traverso (cf. \cite{T89}).
\item An increasing number of cores used during the parallel computation of standard bases resp. 
          associated primes speeds up the computation if the corresponding problem in positive characteristic
          takes some time to be computed. If the computations in positive characteristic are fast then an increasing
          number of cores may slow down the computations because of too much overhead.
\item In the current implementation Chinese remainder and Farey fractions are not parallelized. Experiments
          (e.g. the computation of the Gr\"obner basis of \texttt{Cyclic\_9}) show that the computations in positive
          characteristic need different time on different cores. Therefore one should apply Chinese remainder and
          Farey fractions already to partial results.
\item For zero--dimensional primary decomposition the modular approach is very efficient. This
          should be extended to higher--dimensional ideals.  
\end{enumerate}
\end{con}

\section{Acknowledgement}

The authors would like to thank Wolfram Decker and Gert-Martin Greuel for helpful discussions to prove 
Theorem \ref{thmModStd}. We also thank Christian Eder for important hints and discussions concerning the 
implementation of the described algorthms. In addition, we thank Frank Seelisch for revealing the observation 
about \textsc{Magma V2.16--11} as specified in section \ref{secIntro}. Finally, we would like to thank the 
anonymous referees whose comments greatly improved the paper.

\end{document}